\documentclass[11pt,epsfig]{article}

\usepackage{fullpage,amsmath,amsfonts,amsthm,graphicx,amssymb,textcomp,enumerate,multicol, thmtools, thm-restate}
\graphicspath{{C:/Users/Gabe/Desktop/Math Stuff/Research/Comparing (g,k)-Trisections/Figures}}

\newcommand{\Lim}[1]{\raisebox{0.5ex}{\scalebox{0.9}{$\displaystyle \lim_{#1}\;$}}} 

\newtheorem{theorem}{Theorem}[section]
\newtheorem{corollary}[theorem]{Corollary}
\newtheorem{lemma}[theorem]{Lemma}
\newtheorem{proposition}[theorem]{Proposition}

\theoremstyle{definition}
\newtheorem{definition}[theorem]{Definition}
\newtheorem{question}[theorem]{Question}

\theoremstyle{remark}
\newtheorem{remark}[theorem]{Remark}

\title{Comparing 4-Manifolds in the Pants Complex via Trisections}
\author{Gabriel Islambouli}

\begin{document}

\maketitle

\begin{abstract}
Given two smooth, oriented, closed 4-manifolds $M_1$ and $M_2$, we construct two invariants, $D^P(M_1, M_2)$ and $D(M_1, M_2)$, coming from distances in the pants complex and the dual curve complex respectively. To do this, we adapt work of Johnson on Heegaard splittings of 3-manifolds to the trisections of 4-manifolds introduced by Gay and Kirby. Our main results are that the invariants are independent of the choices made throughout the process, as well as interpretations of ``nearby" manifolds. This naturally leads to various graphs of 4-manifolds coming from unbalanced trisections, and we briefly explore their properties. 
\end{abstract}

\section{Introduction}
	
In \cite{JJ}, Johnson uses two closely related simplicial complexes associated to surfaces in order to define invariants of 3-manifolds. In particular, using Heegaard splittings, \cite{JJ} defines distances between two 3-manifolds in the pants complex and the dual curve complex, which are independent of the particular Heegaard splittings chosen. An interesting interpretation of the distance between $M_1$ and $M_2$ in the dual curve complex is that it is equal to the minimum number of components of a link $L \subset M_1$ so that Dehn surgery along $L$ produces $M_2$. 

	Through the recent work of Gay and Kirby in \cite{GK}, trisections of 4-manifolds have arisen as an analogue to Heegaard splittings. A $(g,k)$-trisection of a 4-manifold $X$ is a decomposition into 3 pieces, $X= X_1 \cup X_2 \cup X_3$, where each $X_i$ is diffeomorphic to $ \natural ^k S^1 \times D^3$ and is equipped with a genus $g$ splitting of $\# ^kS^1 \times S^2$ on its boundary. The intersection of the three pieces, $X_1 \cap X_2 \cap X_3$, is a surface of genus g called the \textbf{trisection surface}. One of the nice features of a trisection is that it encodes all of the data of a smooth 4-manifold as curves on the trisection surface. We are therefore able to access the numerous complexes associated to surfaces to address questions about 4-manifolds.
	
	We seek to adapt the work of \cite{JJ} to 4-manifolds through trisections.  One of the key observations which allows us to do this is that if we have two 4-manifolds equipped with a $(g,k)$-trisection for the same $g$ and $k$, we may cut out a chosen $X_i$ from each of them, and glue them together in a way which respects the Heegaard splitting on the boundary of $X_i$. This gives a way to view all relevant curves on a single surface, and hence compare them in the chosen complex. This allows us to define two non-trivial distances between trisections: $D(T_1,T_2)$ and $D^P(T_1,T_2)$. 
	
If $T$ is a genus $h$ trisection, and $g=h+3n$ for $n \in \mathbb{N}$, there is a natural way to construct a genus $g$ trisection of the same 4-manifold, which we call a stabilization of $T$ and denote $T^g$. The main theorem of the paper is the following:

\begin{restatable*}{theorem}{welldefined}
\label{thm:wellDefined}
Let $M_1$ and $M_2$ have trisections $T_1$ and $T_2$ respectively. Then the limit $\Lim{g \to \infty}D(T_1^g, T_2^g)$ is well defined and depends only on the underlying manifolds, $M_1$ and $M_2$.

\end{restatable*}

We also show the analogous result in the pants complex, and this allows us to define two natural number valued invariants of two 4-manifolds, $D(M_1,M_2)$ and $D^P(M_1,M_2)$. Sections 3 and 4 are dedicated to exploring properties of these invariants. In Section 3, we find various upper and lower bounds. For example, if $\sigma(M)$ denotes the signature of M, we obtain the following inequality: 

\begin{restatable*}{proposition}{bound}
\label{prop:bound}
$D(M_1,M_2) \geq \frac{1}{2}|\sigma(M_1) - \sigma(M_2)|$.

\end{restatable*}

Section 4 consists of interpretations of nearby manifolds in terms of Kirby calculus. We first show that  manifolds which are close in the pants complex have very similar Kirby diagrams. More precisely, we show the following:

\begin{restatable*}{theorem}{pantsDistanceOne}
\label{thm:PantsDistanceOne}
If $D^P(M_1, M_2)=1$, then $M_1$ and $M_2$ have Kirby diagrams which are identical, except for the framing on some 2-handle.
\end{restatable*}

We also show that manifolds with similar Kirby diagrams are close in the pants complex, which is encompassed in the following theorem:

\begin{restatable*}{theorem}{framingPants}
\label{thm:FramingPants}
Let $M_1$ and $M_2$ be non-diffeomorphic 4-manifolds with the same Euler characteristic which have Kirby diagrams $K_1$ and $K_2$ respectively. If $K_1$ and $K_2$ only differ in the framing of some 2-handle, where the framing differs by 1, then $D^P(M_1,M_2) = 1$.
\end{restatable*}

Our line of inquiry in constructing these invariants leads naturally to the construction of various graphs of 4-manifolds coming from subgraphs of the pants complex and the dual curve complex. Section 5 is dedicated to carefully defining these graphs and obtaining some connectivity results.

\begin{center} \section*{Acknowledgements} \end{center}

The author would like to thank David Gay and Alex Zupan for the helpful conversation which made the author aware of \cite{JJ}. The author would particularly like to thank his advisor, Slava Krushkal, for his many helpful comments and patience throughout the preparation of this paper. This research was supported in part by NSF grant DMS-1612159.
\newline

\subsection{Simplicial Complexes Associated to Surfaces}

The most commonly used complex associated to a surface is the curve complex. It has proven to be a useful tool in investigating the structure of the mapping class group of an orientable surface. We recall the definition here:

\begin{definition}
Given a closed, orientable surface of genus $g \geq 2$, $\Sigma$, the \textbf{curve complex} of $\Sigma$, denoted $C(\Sigma)$, is a simplicial complex built out of simple closed curves on $\Sigma$. Each isotopy class of essential simple closed curves corresponds to a vertex. A collection of n vertices spans an (n-1)-simplex if the corresponding curves can be isotoped to be pairwise disjoint.
\end{definition}

In his seminal work in \cite{JH}, Hempel used the curve complex to give an invariant of Heegaard splittings generalizing the notions of reducibility, weak reducibility, and the disjoint curve property. While Hempel's distance is an indispensable tool for investigating the structure of Heegaard splittings of a 3-manifold, it is unlikely to be useful for constructing invariants of manifolds. This is due to the fact that the invariant completely collapses when a Heegaard splitting is stabilized. Our set up for trisections will have similar problems, so we consider the dual of the curve complex. 

\begin{definition}
Given a closed, orientable surface of genus $g \geq 2$, the \textbf{dual curve complex} of $\Sigma$, denoted $C^*(\Sigma)$, is the simplicial complex whose vertices correspond to maximal dimensional simplices of $C(\Sigma)$. Two vertices in $C^*(\Sigma)$ have an edge between them if the corresponding maximal dimensional simplices of $C(\Sigma)$ share a codimension 1 face.
\end{definition}

For a closed, orientable surface of genus $g \geq 2$, maximal dimensional simplices in $C(\Sigma)$ are of dimension $3g-4$ and correspond to a set $3g-3$ simple closed curves whose union separates the surface into pairs of pants. An edge in the dual curve complex therefore corresponds to starting with one pants decomposition of a surface and replacing one curve in order to obtain another pants decomposition of the surface. If instead of allowing arbitrary curve replacements, we insist that curves are replaced in the simplest way possible, we obtain the pants complex.

\begin{definition}
Given a surface $\Sigma$, the \textbf{pants complex} of $\Sigma$, denoted $P(\Sigma)$, is the simplicial complex whose vertices correspond to isotopy classes of pants decompositions of $\Sigma$. Two vertices $v$ and $v'$ in $P(\Sigma)$ are connected by an edge if the corresponding pants decompositions only differ in one curve, and the two different curves intersect minimally. That is, if $l \in v$ and $l'\in v'$ with $l \neq l'$, then either $l$ and $l'$ lie on a punctured torus with $|l \cap l'| = 1$ or $l$ and $l'$ lie on a four punctured sphere with $|l \cap l'| = 2$. In the case that  $l$ and $l'$ lie on a punctured torus, we say that $v$ and $v'$ are related by an \textbf{S-move}. If $l$ and $l'$ lie on a four punctured sphere we say that $v$ and $v'$ are related by an \textbf{A-move}. See figure \ref{fig:PantsComplexMoves} for an illustration of these moves.
\end{definition}

\begin{figure}
\centering
\includegraphics[scale=.3]{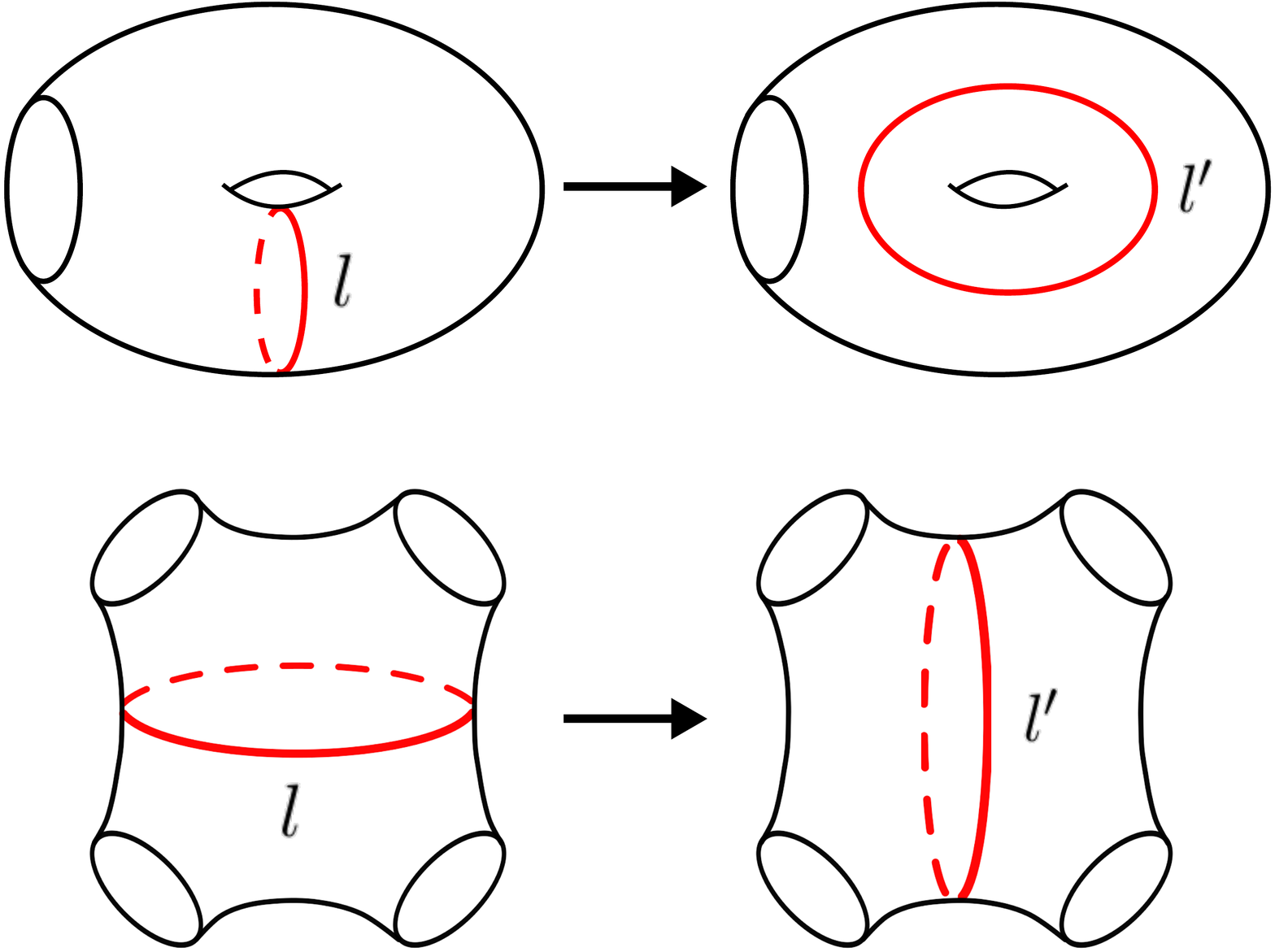}
\caption{Above: An S-move in the pants complex. Below: An A-move in the pants complex.}
\label{fig:PantsComplexMoves}
\end{figure}

There is a natural map from the pants complex into the dual curve complex which is bijective on vertices and injective on edges. In \cite{HT}, Hatcher and Thurston prove that the pants complex is connected and so the aforementioned map shows that the dual curve complex is connected. We therefore get naturally defined metrics on the 1-skeleton of these complexes.

\begin{definition}
Let $v_1$ and $v_2$ be two vertices in $C^*(\Sigma)$. The \textbf{dual distance}, $D(v_1,v_2)$ is the length of the minimal path between $v_1$ and $v_2$ in the dual curve complex. Similarly if $v_1$ and $v_2$ are two vertices in $P(\Sigma)$, the \textbf{pants distance} $D^P(v_1,v_2)$ is the length of the minimal path between $v_1$ and $v_2$ in the pants complex.
\end{definition}

Since the pants complex appears as a subcomplex of the dual curve complex, we get the inequality $D^P(v_1,v_2) \geq D(v_1,v_2)$. This inequality should be kept in mind when bounds are discussed later in the paper. 

\subsection{Trisections}

To fix notation, we briefly summarize the relevant notions in the theory of trisections of 4-manifolds. For a more detailed account, the reader is referred to \cite{GK}.

\begin{definition}
A $\mathbf{(g,k})$\textbf{-Trisection} of a smooth, closed 4-manifold $M$ is a decomposition $M = X_1 \cup X_2 \cup X_3$ such that: 
\begin{itemize}
\item $X_i \cong \natural ^kS^1 \times D^3$
\item $X_i \cap X_j = H_{ij}$ is a genus g handlebody.
\item $\partial X_i = H_{ij}\cup H_{ik}$ is a genus g Heegaard splitting for  $\partial X_i \cong \# ^kS^1 \times S^2$
\item $X_1 \cap X_2 \cap X_3$ is a closed, orientable, genus g surface.
\end{itemize}
\end{definition}
\vspace{.05in}

In \cite{GK}, Gay and Kirby show that every smooth, closed, 4-manifold admits a trisection. At times, it will be useful to relax the condition that all of the $X_i$ are diffeomorphic to the same 4-dimensional handlebody. In particular, we allow $X_i \cong \natural ^{k_i} S^1 \times D^3$ where for $i \neq j$ it is possible that $k_i \neq k_j$. In this case, we insist that  $\partial X_i = H_{ij}\cup H_{ik}$ is a genus g Heegaard splitting for  $\partial X_i \cong \# ^{k_i} S^1 \times S^2$. We will call this more general setup an \textbf{unbalanced} $\mathbf{(g;k_1,k_2,k_3)-trisection}$. Note that $S^4$ has unbalanced trisections with parameters $(1;1,0,0)$, $(1;0,1,0)$, and $(1;0,0,1)$. These can be used to balance trisections by taking connected sums (see definition 3.8 of \cite{MSZ} for more details on this construction). Unless otherwise noted, all trisections will be assumed to be balanced.

The union $H_{12} \cup H_{23} \cup H_{31}$ is called the \textbf{spine} of the trisection. Note that if we thicken the spine of the trisection by taking the product of the surface with $D^2$ and the handlebodies with $D^1$, then we are left with a 4-manifold with 3 boundary components each diffeomorphic to $\#^{k}S^1 \times S^2$. Recovering the original 4-manifold amounts to gluing back in 3 copies of $\natural^k S^1 \times D^3$. By \cite{LP} this can only be done in one way. In other words, the spine uniquely determines the trisection. The spine, in turn, is uniquely determined by 3 cut systems for the handlebodies which pairwise form Heegaard diagrams for  $\#^{k}S^1 \times S^2$. Thus, a trisected 4-manifold is completely determined by these cut systems drawn on the trisection surface $\Sigma$. We refer to the trisection surface, together with the 3 cut systems as a \textbf{trisection diagram}.

If $T_1$ is a $(g_1,k_1)$-trisection, and $T_2$ is a $(g_2,k_2)$-trisection, we may form their connected sum $T_1 \# T_2$, which inherits the structure of a $(g_1+g_2,k_1+k_2)$-trisection. On the level of diagrams, this amounts to taking the connected sum of trisection diagrams for $T_1$ and $T_2$. $S^4$ has a $(3,1)$ trisection so that if $T$ is a $(g,k)$-trisection for $M^4$, we may form a $(g+3,k+1)$-trisection for $M$ by taking a connected sum with the aforementioned trisection for $S^4$. The resulting $(g+3,k+1)$-trisection is called a \textbf{stabilization} of T. Let T be a genus h trisection and let $g=h+3n$ for some $n \in \mathbb{N}$. We denote by $\mathbf{T^g}$ the $(h+3n,k+n)$-trisection obtained by stabilizing $T$ $n$ times to a genus $g$ trisection. If $T$ has spine $H_{12} \cup H_{23} \cup H_{31}$ and trisection surface $\Sigma$, we will denote the spine and trisection surface of $T^g$ by $H_{12}^g \cup H_{23}^g \cup H_{31}^g$ and $\Sigma^g$. The following theorem will be essential for extending invariants of trisections to invariants for 4-manifolds. It can be seen as the analogue to the Reidemeister-Singer theorem for trisections.

\begin{theorem} (Theorem 11 of \cite{GK}) 
\label{thm:TrisectionsStabilize}
If $T_1$ and $T_2$ are trisections of the same manifold, then there exists a natural number $n$ so that $T_1^n$ and $T_2^n$ are isotopic as trisections. That is, if $T_1^n = X_1 \cup X_2 \cup X_3$ and $T_2^n= Y_1 \cup Y_2 \cup Y_3$ then there exists a map $f$ isotopic to the identity so that $f(X_i)=Y_i$.
\end{theorem}

\section{Distances of Trisections}

Let $(T_1, \Sigma_1)$ and $(T_2, \Sigma_2)$ be two $(g,k)-$trisections with corresponding spines $H_{\alpha1} \cup H_{\beta1} \cup H_{\gamma1}$ and $H_{\alpha2} \cup H_{\beta2} \cup H_{\gamma2}$. Both  $H_{\alpha1} \cup H_{\beta1}$ and $H_{\alpha2} \cup H_{\beta2}$ are genus g Heegaard splittings for $\#^k S^1 \times S^2$. Waldhausen's theorem \cite{FW} therefore asserts that both of these are in fact the unique genus g splitting of $\#^k S^1 \times S^2$. Therefore, there exists a map $\phi: H_{\alpha1} \cup H_{\beta1}  \to  H_{\alpha2} \cup H_{\beta2}$ so that $\phi(H_{\alpha1}) = H_{\alpha2}$ and $\phi(H_{\beta1}) = H_{\beta2}$. Such a map induces isometries on the various complexes associated to $\Sigma_1$ and $\Sigma_2$, and we will denote the induced isometry by $\hat{\phi}$.

If we fix an identification of both $H_{\alpha1} \cup H_{\beta1}$ and $H_{\alpha2} \cup H_{\beta2}$ with $\#^k S^1 \times S^2$, all such maps (up to isotopy with $\phi_t(\Sigma_1) = \Sigma_2$) can be identified with the mapping class group of the Heegaard splitting which we denote $\mathbf{Mod(\#^k S^1 \times S^2, \Sigma_g)}$. The mapping class groups of Heegaard splittings have been studied extensively and can be quite complicated. For example, when $g>k$, the group $Mod(\#^k S^1 \times S^2, \Sigma_g)$ will always have pseudo-Anosov elements \cite{JR}.


We say a vertex $v \in C^*(\Sigma)$ (or $P(\Sigma)$) \textbf{defines a handlebody}, H, if all of the curves in the pants decomposition corresponding to v bound disks in H. Equivalently, attaching 3 dimensional 2-handles to $\Sigma$ along the curves of v and filling in the resulting 2 sphere boundary components with 3 balls produces H. We are now ready to define the main objects of study.

\begin{definition}
Let $M_1$ and $M_2$ be two 4-manifolds equipped with $(g,k)-$trisections $T_1$ and $T_2$. The \textbf{dual distance} between $T_1$ and $T_2$, $D(T_1, T_2)$, is \newline
 $min \{D(\hat{\phi}(v_1),v_2)|$ $v_1$ defines  $H_{\gamma1}$, $v_2$ defines $H_{\gamma2}\}$. Similarly, the \textbf{pants distance}, $D^P(T_1, T_2)$, is \newline
 $min \{D^P(\hat{\phi}(v_1),v_2)|$ $v_1$ defines  $H_{\gamma1}$, $v_2$ defines $H_{\gamma2}\}$.
\end{definition}

\begin{figure}
\centering
\includegraphics[scale=.44]{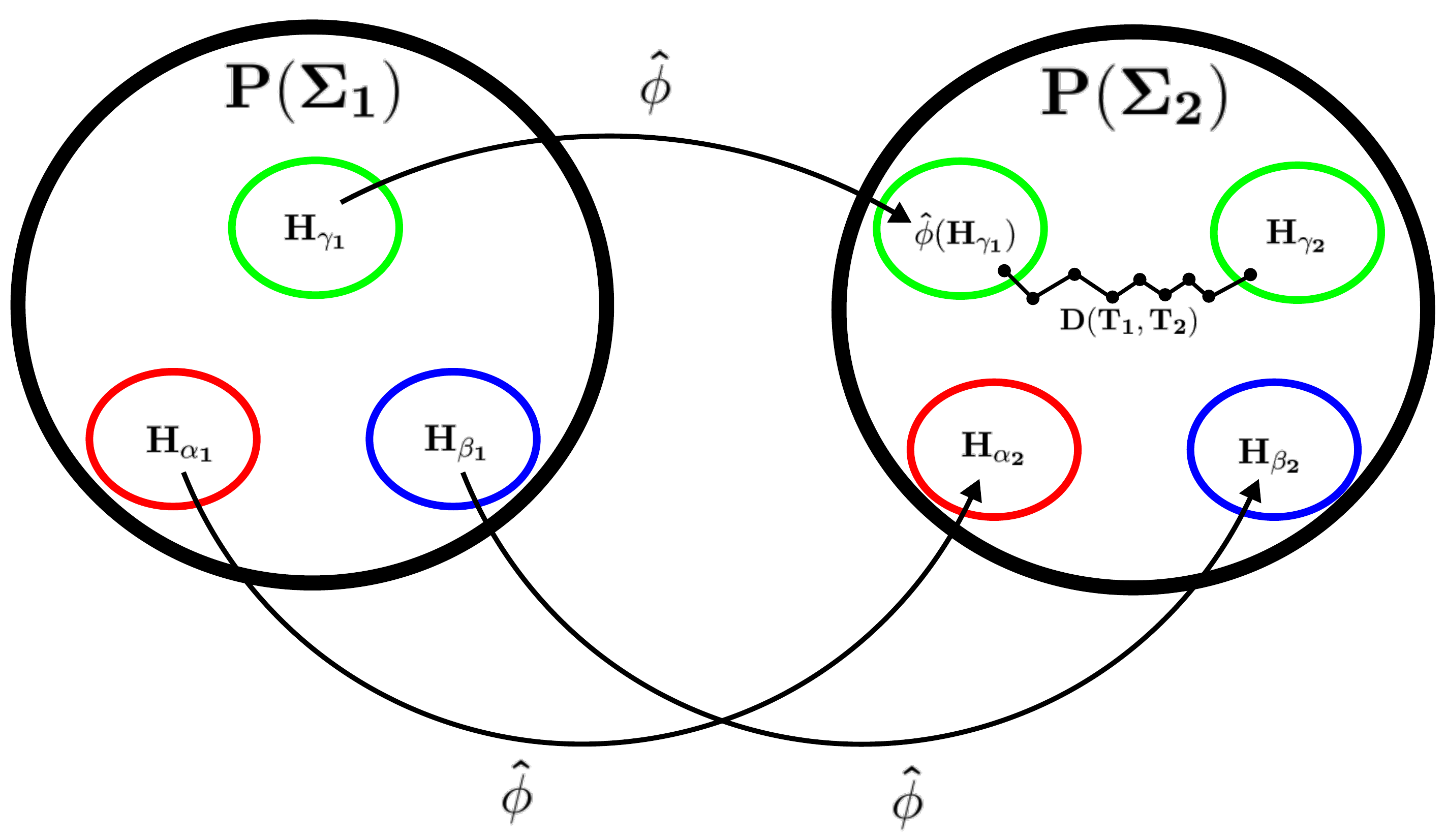}
\caption{The distance between two trisections is the minimum distance between the sets $\hat{\phi}(H_{\gamma1})$ and $H_{\gamma2}$}
\label{fig:DistancePicture}
\end{figure}

Here, the minimum is taken over all orientation preserving maps $\phi: H_{\alpha1} \cup H_{\beta1}  \to  H_{\alpha2} \cup H_{\beta2}$ so that $\phi(H_{\alpha1}) = H_{\alpha2}$ and $\phi(H_{\beta1}) = H_{\beta2}$ as well as all vertices defining the handlebodies in the respective complexes. See Figure \ref{fig:DistancePicture} for an illustration of the definition. Since these distances are natural number valued, they give well defined invariants of two $(g,k)$-trisections. Furthermore, if either distance is $0$, then the distance minimizing map extends to a homeomorphism of spines, which means that $T_1$ and $T_2$ are in fact diffeomorphic trisections. Since there are many manifolds admitting a $(g,k)$-trisection for a given g and k with $g \neq k$, this distance is nontrivial.

We now seek to extend these distances of particular trisections to well defined distances of 4-manifolds. To do this, we need to understand how the distance behaves under stabilization. If $(T, \Sigma)$ is a trisection, we may stabilize $T$ by puncturing $\Sigma$ in a disk and gluing on the stabilizing surface shown in Figure \ref{fig:PuncturedStabilizingSurface}. From this point of view, it is clear that we should begin by understanding paths in the complexes associated to $\Sigma \backslash D^2$. The following key lemma treating this case is contained in Lemma 15 of \cite{JJ}.

\begin{lemma} 
\label{lem:PathInPuncturedSurface}
Let $v_1,v_2,....v_n$ be a minimal path in $C^*(\Sigma)$ (or $P(\Sigma)$) between two handlebodies $H_1$ and $H_2$. Then there exists a disk $D \subset \Sigma$ and a path $v_1', v_2',... v_m'$ on $C^*(\Sigma  \backslash D)$ (respectively $P(\Sigma  \backslash D))$ with $m \leq 2n$ so that after capping off $\Sigma \backslash D$ with a disk, every loop in  $v_1'$ bounds a disk in $H_1$, and every loop in $v_m'$ bounds a disk in $H_2$. Moreover, if there is some loop which is never moved in the path from $v_1$ to $v_n$, then there exists a disk $D$ and a path $v_1', v_2',... v_m'$ satisfying the previous conclusions with $m=n$.
\end{lemma}

In what follows, we will adapt the work of \cite{JJ} to prove that the invariants of trisections behave well under stabilization. In order to aid exposition, we will only explicitly treat the case of the distance in the dual curve complex. It should be clear after the fact that by using the part of Lemma \ref{lem:PathInPuncturedSurface} pertaining to the pants complex, all of the arguments go through virtually unchanged.

\begin{figure}
\centering
\includegraphics[scale=.4]{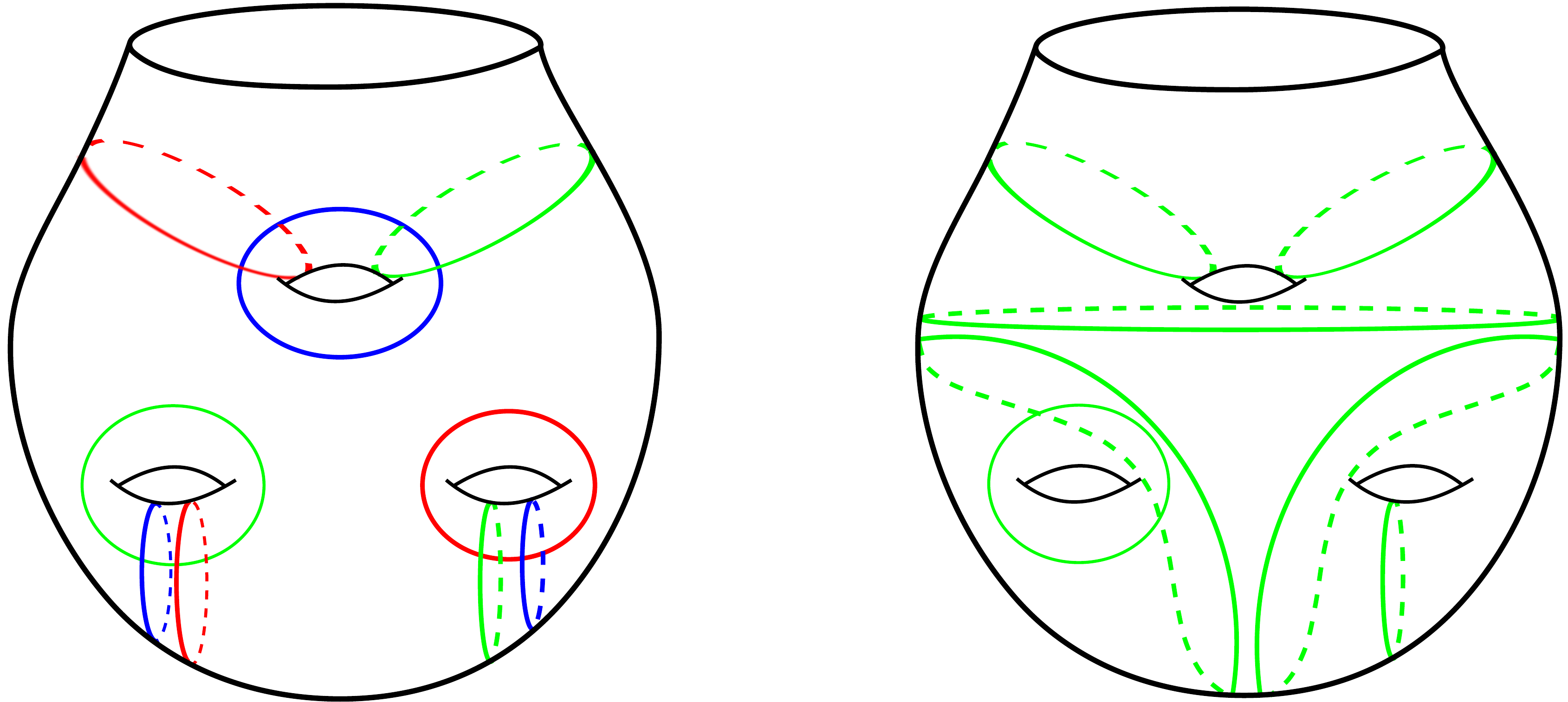}
\caption{Left: Stabilizing a trisection amounts to gluing this diagram onto a punctured trisection diagram. Right: A pants decomposition for one handlebody of the stabilizing surface}
\label{fig:PuncturedStabilizingSurface}
\end{figure}

\begin{lemma}
Let $(T_1,\Sigma_1)$ and $(T_2,\Sigma_2)$ be $(g,k)$-trisections, and let $T_1^h$ and $T_2^h$ be their genus $h$ stabilizations. Then  $D(T_1^h,T_2^h) \leq 2D(T_1,T_2)$.
\end{lemma}

\begin{proof}
 If $D(T_1,T_2) = n$ then there is a map $\phi: H_{\alpha1} \cup H_{\beta1}  \to  H_{\alpha2} \cup H_{\beta2}$ and a path $v_1,v_2,....v_n$ in $C^*(\Sigma_2)$ so that $\hat{\phi}(v_1)$ defines $\phi(H_{\gamma_1})$ and $v_n$ defines $H_{\gamma_2}$. Let $l_i^1...l_i^m$ be the loops corresponding to the pants decomposition given by $v_i$. Let $D$ be a disk in the annulus formed by two parallel copies of $l_1^1$. Consider the pants decomposition for $\Sigma_2 \backslash D$ consisting of $l_1^1,...l_1^m,l_1^{m+1}$ where $l_1^{m+1}$ is the parallel copy of $l_1^1$ on $\Sigma_2$ lying on the other side of $D$ on $\Sigma_2 \backslash D.$ Let $v_1'$ be the corresponding vertex of $C^*(\Sigma_2 \backslash D)$
 
By Lemma \ref{lem:PathInPuncturedSurface}, there is a path from $v_1'$ to a vertex $w$ such that if $l \in w$, then after capping off $\Sigma_2 \backslash D$ with a disk, $l$ is isotopic to some loop in $v_n$. Moreover, this path is  of length at most $2n$.

We first treat the case of a single stabilization. Consider the stabilization of $\Sigma_1$ produced by cutting out the disk $\phi^{-1}(D)$ gluing on a stabilizing surface to the resulting boundary component. Then $H_{\gamma_1}^{g+3}$ has a pants decomposition given by the curves in $\phi^{-1}(v_1')$ along with $\phi^{-1}(\partial D)$ and the pants decomposition for the stabilizing surface shown in Figure \ref{fig:PuncturedStabilizingSurface}. By gluing on a stabilizing surface to $\partial D$, we can extend $\phi$ to a map $\phi^{g+3}: \Sigma_1^{g+3} \to \Sigma_2^{g+3}$ such that $\phi^{g+3}(w) \cup \partial D \cup \phi^{g+3}$(the curves shown in Figure \ref{fig:PuncturedStabilizingSurface}) is a pants decomposition for $H_{\gamma 2}^{g+3}$. Since the path in $C^*(\Sigma_1 \backslash D)$ from $v_1'$ to $w$ takes place away from the stabilizing surface it corresponds to a path in $C^*(\Sigma_2^{g+3})$ so that $D(T_1^{g+3},T_2^{g+3}) \leq 2D(T_1,T_2)$. To achieve the more general result, simply connect sum multiple stabilizing surfaces first before connect summing with the given trisection surfaces.
\end{proof}

\begin{lemma}
For sufficiently large g, $D(T_1^h,T_2^h) \leq  D(T_1^g,T_2^g)$ when $h \geq g$.
\end{lemma}

\begin{proof}
By the previous lemma, $D(T_1^g,T_2^g) \leq 2D(T_1,T_2)$. Choose g so that $3g-3 > 2D(T_1,T_2)$. Since a pants decomposition of $\Sigma_2^g$ consists of $3g-3$ loops, it follows that some loop is never moved in the path on $\Sigma^g$. In this case, we conclude by Lemma \ref{lem:PathInPuncturedSurface} that paths on $\Sigma^g$ from $\hat{\phi^g}(H^g_{\gamma1})$ to $H^g_{\gamma2}$ lift to paths of the same length on $\Sigma^h$ from  $\hat{\phi^h}(H^h_{\gamma1})$ to $H^h_{\gamma2}$.
\end{proof}

\welldefined

\begin{proof}
Since the sequence $D(T_1^g, T_2^g)$ is natural number valued and non-increasing for sufficiently large g, it converges. Furthermore, by Theorem \ref{thm:TrisectionsStabilize}  any two trisections of the same manifold have a common stabilization fixing the labels of the handlebodies. Therefore, if $T_1$ and $T_3$ are distinct trisections of $M_1$, and $T_2$ and $T_4$ are distinct trisections of $M_2$ then there exists an $h$ so that $T_1^h$ is isotopic to $T_3^h$ and $T_2^h$ is isotopic to $T_4^h$. Then for $g>h$ we have that $D(T_1^g, T_2^g) = D(T_3^g, T_4^g)$ so that $\Lim{g \to \infty}D(T_1^g, T_2^g) = \Lim{g \to \infty}D(T_3^g, T_4^g)$.
\end{proof}

\begin{remark}
The reader may be concerned that the definition seems to distinguish between which third of the trisection is labeled $X_1$, whereas we have seemingly defined an invariant of a 4-manifold which is not sensitive to this information. However, Theorem \ref{thm:wellDefined} does actually encompass this case. Suppose $M_1$ has two trisections of the form $T_1 = (X_1,X_2,X_3)$ and $T_2 = (Y_1,Y_2,Y_3)$ such that $Y_i= X_{i-1}$ with indices taken mod 3. In \cite{GK}, it is shown that any two trisections of the same manifold have a common stabilization fixing the labels of the sectors. We therefore have a map $f:M^4 \to M^4$ isotopic to the identity so that $f(Y_i^h) = X_i^h$.
\end{remark}

We are now justified in making the following definitions.

\begin{definition}
Let $M_1$ and $M_2$ be two 4-manifolds which have $(g,k)$-trisections for the same $g$ and $k$. The \textbf{dual distance} between $M_1$ and $M_2$ is $\Lim{g \to \infty}D(T_1^g, T_2^g)$ where $T_1$ is any trisection of $M_1$ and $T_2$ is any trisection of $M_2$ with the same parameters as $T_1$. Similarly, the \textbf{pants distance} between two 4-manifolds is $\Lim{g \to \infty}D^P(T_1^g, T_2^g)$.
\end{definition}

\begin{remark} 
It should be noted that in the 3-manifold case, any two pants decompositions will determine a 3-manifold, so that minimal paths between two 3-manifolds pass through intermediary 3-manifolds. This nice property simplifies many of the arguments in \cite{JJ}. In our set up, we can not guarantee that $H_{\alpha2}$, $H_{\beta2}$, and the handlebody determined by an intermediary vertex in a minimal path between $\hat{\phi}(H_{\gamma1})$ and $H_{\gamma2}$ will still pairwise form Heegaard splittings for $\#^{k} S^1 \times S^2$. Therefore, these three handlebodies may not form the spine of a trisection, and there may be no way to uniquely obtain a closed 4-manifold from this information. This leads to two natural questions: 1) Can we pass between trisections though paths whose intermediary vertices form trisections? 2) Is there any significance to the 3 handlebodies which occur in paths between two trisections?
\end{remark}

It is clear from the definitions that trisections, $T_1$ and $T_2$, can only be compared when $(g_1,k_1) = (g_2,k_2)$. However, it is not immediately obvious when two manifolds can be compared. A necessary and sufficient condition for comparing $M_1$ and $M_2$ is that both manifolds have a $(g,k)$-trisection for some $g$ and $k$. 
If a 4-manifold has a $(g,k)$-trisection, the Euler characteristic is given by $\chi(M)= 2+g-3k$, so it is necessary that $\chi(M_1) = \chi(M_2)$. The following straightforward lemma shows that this is also a sufficient condition.

\begin{lemma}
$D(M_1,M_2)$ and $D^P(M_1,M_2)$ are well defined whenever $\chi(M_1) = \chi(M_2)$.
\end{lemma}

\begin{proof}
Let $M_1$ have a $(g_1, k_1)$-trisection, $T_1$, and let $M_2$ have a $(g_2, k_2)$-trisection, $T_2$. Now since $\chi(M_1) = \chi(M_2)$, $2+g_1-3k_1 = 2+g_2-3k_2$. Without loss of generality, assume $k_1>k_2$. Then by stabilizing $T_2$ $(k_1 - k_2)$ times we get a new trisection of $M_2$, $T_2'$ with $k_2' = k_1$ and $g_2' = g_2+3(k_1-k_2) = g_2+(g_1-g_2) = g_1$, hence these two trisections can be compared.
\end{proof}

It is natural to ask whether these distances induce a metric on the set of 4-manifolds with the same Euler characteristic. It follows quickly from the definition that these distances are 0 if and only if the manifolds are diffeomorphic. These distances are also symmetric, since if $\phi$ is the distance minimizing map for $D(M_1,M_2)$ (or $D^P(M_1,M_2)$), then $\phi^{-1}$ minimizes the distance  $D(M_2,M_1)$ (respectively $D^P(M_2,M_1)$). The triangle inequality, however, is likely false. This is due to the fact that we are minimizing over handlebody sets of infinite diameter in the complexes. Given three handlebodies $H_1$, $H_2$, and $H_3$, the representatives of $H_2$ closest to $H_1$ and may be far away from the representative of $H_2$ closest to $H_3$.

\section{Some Bounds}

\begin{lemma}
\label{connectSums}
If $D(M_1, M_2) = n$ and $D(M_3, M_4) = m$, then $D(M_1 \# M_3, M_2 \# M_4) \leq n+m$.
\end{lemma}

\begin{proof}
Stabilize trisections of $M_1$ and $M_2$ to genus $g$ trisections, $(T_1^g, \Sigma_1^g)$ and $(T_2^g, \Sigma_2^g)$, with $3g-3 > n$. Also, stabilize trisections of $M_3$ and $M_4$ to genus $h$ trisections,$(T_3^h, \Sigma_3^h)$ and $(T_2^h, \Sigma_2^h)$, with $3h-3 > m$. Since a pants decomposition for $\Sigma_2^g$ has $3g-3$ loops, it follows that some loop in the path from $\hat{\phi}^g(H_{\gamma1}^g)$ to $H_{\gamma2}^g$ is never moved, where $\hat{\phi}$ is the distance minimizing map. Let $v_1', v_2',...,v_n'$ be the path in the dual curve complex (or the pants complex) guaranteed by Lemma \ref{lem:PathInPuncturedSurface} on $C^*(\Sigma_2^g \backslash D)$, and let  $w_1', w_2',...,w_n'$ be the path guaranteed by the same lemma on $C^*(\Sigma_4^h \backslash D')$. 

Form the connect sum, $\Sigma_2^g \# \Sigma_4^h$ along the disks $D$ and $D'$. Let $v_i' \cup w_j'$ be the pants decomposition of $\Sigma_2^g \# \Sigma_4^h$ consisting of the pants decomposition for  $\Sigma_2^g$ induced by $v_i'$, the pants decomposition for $\Sigma_4^h$ induced by $w_j'$, along with the additional curve $\partial D = \partial D'$. Then the path $v_1' \cup w_1', v_2' \cup w_1',...,v_n' \cup w_1', v_n' \cup w_2'... v_n' \cup w_m'$ is a path of length $n+m$ from $M_1 \# M_3$ to $M_2 \# M_4$.
\end{proof}

\begin{corollary}
\label{char2}
For any N with $\chi(N) = 2,$ $D(M_1 \# N, M_2 )\leq  D(M_1, M_2)+ D(N, S^4)$
\end{corollary}

The question of whether $D(M_1 \# M_3, M_2 \# M_4) = n+m$ is quite easily shown to be false. For example, if $M_1$ and $M_2$ are homeomorphic, but not diffeomorphic, 4-manifolds which become diffeomorphic after a single connected sum with $S^2 \times S^2$, then  $D(M_1, M_2) \neq 0$ whereas $D(M_1 \# S^2 \times S^2, M_2 \# S^2 \times S^2) = 0.$


We next seek to prove a lower bound on the distance between two manifolds based on the difference of their signatures. To this end, we briefly discuss how this information can be recovered from a trisection. Given a genus g surface $\Sigma,$ choose a symplectic basis for $H_1(\Sigma_g, \mathbb{R})$. That is, a basis $\{a_1,b_1,... a_g,b_g\}$ so that for all $i$ and $j$, $|a_i \cap a_j| = |b_i \cap b_j| = 0$ and $|a_i \cap b_j| = \delta_{ij}$, and let $\omega$ be the associated symplectic form on $\mathbb{R}^{2g}$. 

Given a trisection with spine $H_1 \cup H_2 \cup H_3$ we get 3 Lagrangian subspaces of $H_1(\Sigma_g, \mathbb{R})$  given by $L_i = ker(i_*:H_1(\Sigma_g, \mathbb{R}) \to H_1(H_i, \mathbb{R})$. We may define a symmetric bilinear form, q, on $L_1 \oplus L_2 \oplus L_3$ by $q((x_1,x_2,x_3), (y_1,y_2,y_3))= \omega(x_1, y_2) + \omega(y_1, x_2) + \omega(x_2, y_3) + \omega(y_2, y_3) + \omega(x_3, y_1) + \omega(y_3, x_1) + \omega(x_3, y_1)$. In \cite{GK}, it is observed that the signature of the matrix associated to this bilinear form is the signature of the original 4-manifold. While intermediary vertices in minimal paths between handlebodies will not always define trisections, the signature of this matrix will always be well defined. As a result, we obtain the following proposition.

\bound

\begin{proof}
Suppose $D(M_1, M_2) = n$ with a path $v_1, ... v_n$, so that $v_1$ defines $\phi(H_{\gamma1})$ and $v_n$ defines $H_{\gamma2}$. At each vertex, we have a triple of handlebodies determined by $H_{\alpha2}$, $H_{\beta2},$ and the handlebody determined by $v_i$. These in turn determine three Lagrangian subspaces of $H_1(\Sigma_2, \mathbb{R})$: $L_1$, $L_2$, and $L_{v_{i}}$. Going from $v_i$ to $v_{i+1}$ involves changing a single curve in the pants decomposition so that $L_{v_{i}}$ and $L_{v_{i+1}}$ have bases in $H_1(\Sigma_2, \mathbb{R})$ which are the same except for possibly one vector. 

Let $M_i$ be the matrix corresponding to the symmetric bilinear form $q_i$ on $L_1 \oplus L_2 \oplus L_{v_i}$. Then $M_i$ has real eigenvalues $\lambda_1 \leq \lambda_2 ... \leq \lambda_{3g}$. Let $a_1,... a_g$ be a basis for $L_{v_i}$. In going from $L_{v_{i}}$ to $L_{v_{i+1}}$, it is possible that none of the basis vectors are changed, in which case the signature of the matrix is obviously unchanged. It is also possible that one vector, say $a_j$ is changed. Let $M_i'$ be the matrix obtained by deleting the row and column corresponding to $a_j$, and let $\lambda_1' \leq \lambda_2' ... \leq \lambda_{3g-1}'$ be its eigenvalues. By the Cauchy interlacing theorem, $\lambda_1 \leq \lambda_1' \leq \lambda_2 \leq \lambda_2' ... \leq \lambda_{3g-1}' \leq \lambda_{3g}$ so that $|\sigma(M_i) - \sigma(M_i')| \leq 1$. Similarly, we may obtain $M_i'$ by deleting a row and column of $M_{i+1}$ so that $|\sigma(M_{i+1}) -  \sigma(M_i')| \leq 1$. Therefore, $|\sigma(M_{i+1}) - \sigma(M_i)| \leq 2$. The result immediately follows. 
\end{proof}

Comparing the standard $(g,0)-$trisections of $\#^g\mathbb{C}P^2$ and  $\#^g \overline{\mathbb{C}P^2}$ one can see that this bound is sharp. Moreover, we may conclude that $\Lim{g \to \infty}D(\#^g\mathbb{C}P^2, \#^g \overline{\mathbb{C}P^2}) = \infty$.

\section{Nearby Manifolds}

\begin{figure}
\centering
\includegraphics[scale=.4]{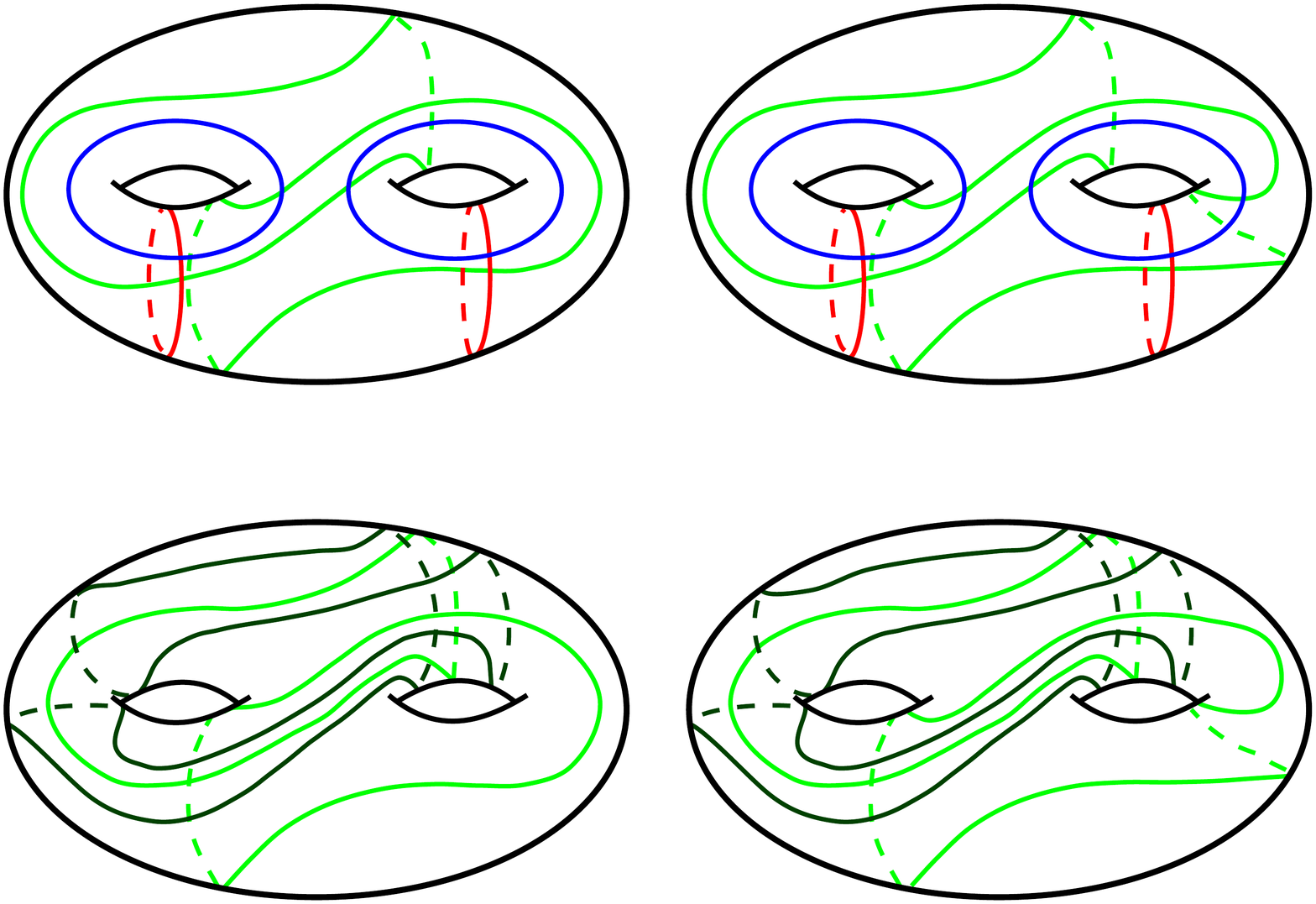}
\caption{$S^2\times S^2$ and $S^2 \widetilde{\times} S^2$  are distance one in the pants complex }
\label{fig:S2xS2DistanceOne}
\end{figure}

We next seek to build some intuition as to what it means when 4-manifolds are close to each other with respect to our distances. We first consider an illustrative example. The top of Figure \ref{fig:S2xS2DistanceOne} shows trisection diagrams $T_1$ for $S^2\times S^2$ and $T_2$ for $S^2 \widetilde{\times} S^2$ where the $\alpha$ and $\beta$ handlebodies are identical. Below that are pants decompositions for the $\gamma$ handlebodies which are identical except for in one curve. The curves which are different intersect exactly once, showing that $D^P(T_1,T_2)=1$. Moreover, we have 2 curves in the pants decomposition which never move, so by Lemma \ref{lem:PathInPuncturedSurface}, the path lifts to new paths of distance one on all stabilizations. Since these manifolds are non-diffeomorphic, we may therefore conclude that $D^P(S^2\times S^2,S^2 \widetilde{\times} S^2)=1$. 

In \cite{GK}, it is shown how to obtain a Kirby diagram from a trisection diagram, and these particular trisection diagrams give rise to Kirby diagrams which are identical except for in the framing of a 2-handle. We seek to show that this is in fact the case in general. That is, if $D^P(M_1, M_2) = 1$ then $M_1$ and $M_2$ have Kirby diagrams which are identical except for in the framing of some  2-handle. To do this, we first consider what it means for two handlebodies to be distance one apart in the pants complex.

\begin{lemma}
\label{DistanceOneHandlebodies}
Let $H_1$ and $H_2$ be two genus g handlebodies with boundary $\Sigma$. If $D^P(H_1, H_2) = 1$ then the manifold $H_1 \cup_\Sigma H_2 \cong \#^{g-1}S^1 \times S^2$. 
\end{lemma}

\begin{proof}
Let $v_1,v_2 \in P(\Sigma)$ define $H_1$ and $H_2$ respectively with $D^P(v_1,v_2) = 1$. The pants decompositions corresponding to these vertices are exactly the same except for some loops $l_1 \in v_1$ and $l_2 \in v_2$. Moreover, since A-moves do not change the handlebody and $H_1 \neq H_2$ we know that $l_1$ and $l_2$ lie in a punctured torus with $|l_1 \cap l_2| = 1$. We may therefore build a Heegaard diagram for $H_1 \cup_\Sigma H_2$  consisting of $g-1$ identical loops in both $v_1$ and $v_2$, along with $l_1$ and $l_2$. It is easy to see that this is a once stabilized splitting for $\#^{g-1}S^1 \times S^2$.
\end{proof}

Genus g Heegaard splittings of $\#^{g-1}S^1 \times S^2$ are in some sense the second most simple Heegaard splittings in a given genus after $\#^{g}S^1 \times S^2$. Genus g trisections where two of the handlebodies form   $\#^{g}S^1 \times S^2$ are easily shown to be diffeomorphic to $\#^g S^1 \times S^3$. Given these facts, it would be reasonable to assume that genus g trisections where two of the handlebodies form  $\#^{g-1}S^1 \times S^2$ are also relatively simple. The following theorem of \cite{MSZ} pertaining to unbalanced trisections makes this precise.

\begin{theorem} (Theorem 1.2 of \cite{MSZ})
\label{(g-1)-Trisections}
Suppose that M admits a $(g;g-1,k_2,k_3)$-trisection $T$, and let $k' = max\{k_2,k_3\}.$ Then M is diffeomorphic to either $\#^{k'}S^1 \times S^3$ or to the connect sum of $\#^{k'}S^1 \times S^3$  with one of either $\mathbb{C}P^2$ or $\overline{\mathbb{C}P^2}$, and $T$ is a connect sum of genus 1 trisections. 
\end{theorem}

%
%


\begin{figure}

  \centering
    \includegraphics[scale=.4]{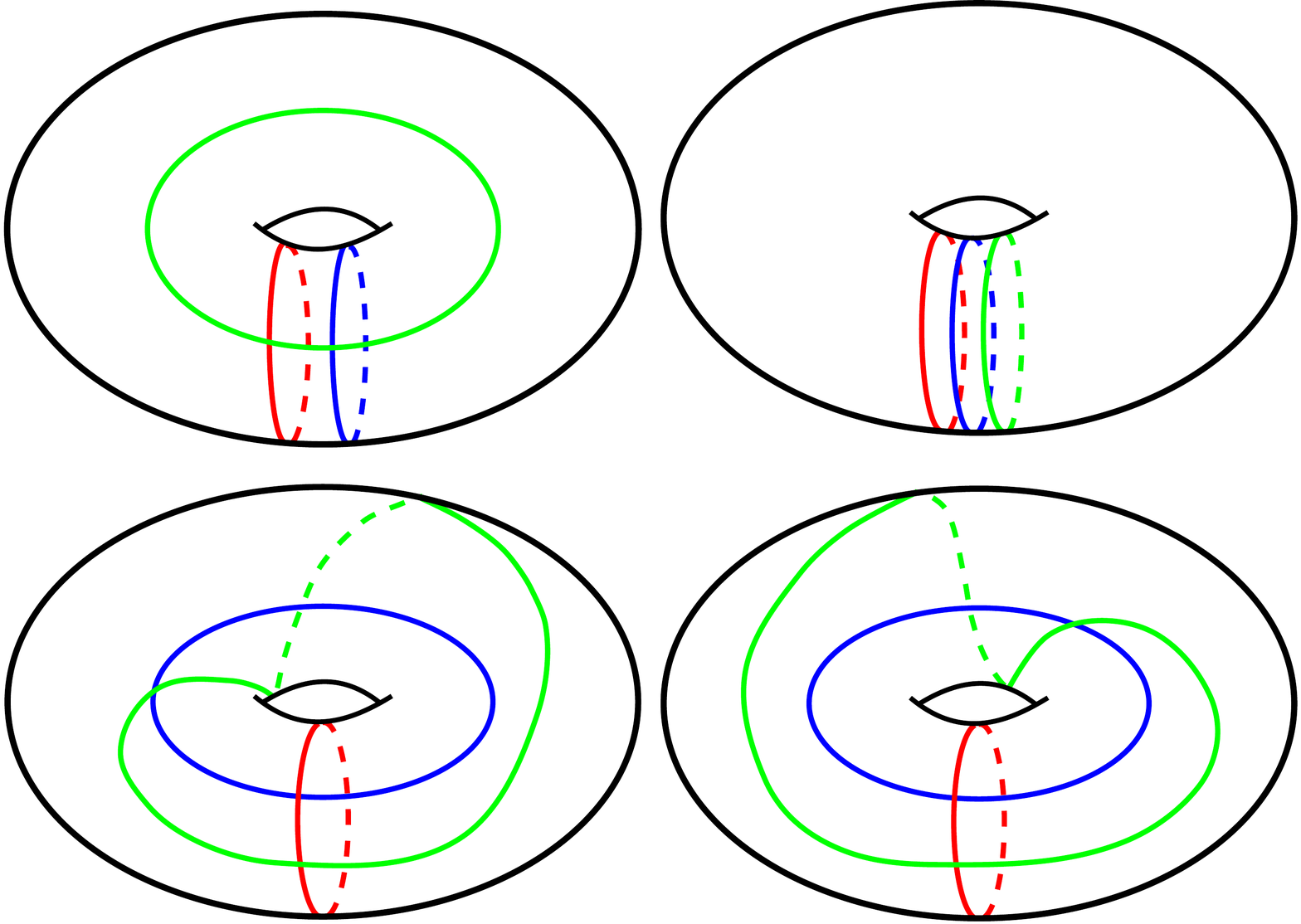}
    \caption{The genus 1 trisections. Top Left: $S^4$. Top Right: $S^1 \times S^3$. Bottom Left: $\mathbb{C}P^2$. Bottom Right: $\overline{\mathbb{C}P^2}$}
    \label{fig:GenusOneTrisections}
    
\end{figure}

\pantsDistanceOne

\begin{proof}
Suppose the distance has stabilized in $T_1$ for $M_1$ and $T_2$ for $M_2$, where $T_1$ and $T_2$ are $(g,k)$-trisections. We will construct 2 new manifolds from these trisections following the schematic found in Figure \ref{fig:DistanceOneArgument}. Consider the manifold obtained by removing $X_1$ from $T_1$ and $X_1'$ from $T_2$, and gluing the resulting two manifolds by the distance minimizing map. Since $D^P(M_1, M_2) = 1$, Lemma \ref{DistanceOneHandlebodies} implies that the 3-manifold $H_{\gamma_1} \cup H_{\gamma_2}$ is diffeomorphic to $\#^{g-1}S^1 \times S^2$. We may cut the resulting 4-manifold along $H_{\gamma_1} \cup H_{\gamma_2}$ to obtain two 4-manifolds each with boundary $\#^{g-1}S^1 \times S^2$. We may fill in each of the resulting manifolds with boundary with $\#^{g-1}S^1 \times D^3$ in order to obtain two trisected, closed 4-manifolds.  

We will focus on the manifold with trisection spine $H_\alpha \cup H_{\gamma_1} \cup H_{\gamma_2}$. This closed 4-manifold inherits the structure of a $(g;g-1,g-k,g-k)$-trisection.  By Theorem \ref{(g-1)-Trisections}, this trisection is a connect sum of genus 1 trisections. In particular, there are curves $l_1...l_g$, all bounding disks in $H_\alpha,$ $ H_{\gamma_1},$ and $H_{\gamma_2}$, which cut $\Sigma$ into $g$ once punctured tori and a sphere with $g$ holes. We may also ensure that each of these tori contain one $\alpha, \gamma_1,$ and $\gamma_2$ curve so that in all but one particular torus, the $\gamma_1$ and $\gamma_2$ curves are identical. 

Let $\alpha^i,$ $\gamma_1^i$, and $\gamma_2^i$ be the three curves on the same punctured torus with $\gamma_1^i \neq \gamma_2^i$. By virtue of the classification of genus 1 trisections, as well as the fact that $\gamma_1^i \neq \gamma_2^i$, these three curves either form a diagram for $\mathbb{C}P^2$ or for $S^4$. However, if both $T_1$ and $T_2$ are balanced trisections, the three curves must form $\mathbb{C}P^2$; for otherwise, $H_\alpha \cup H_\beta$ is $\#^{k}S^1 \times S^2$ whereas $H_\alpha \cup H_{\gamma'}$ is  $\#^{k \pm 1}S^1 \times S^2$. After a diffeomorphism, $\alpha^i,$ $\gamma_1^i$ and $\gamma_2^i$ form the trisection diagram for $\mathbb{C}P^2$ shown in Figure \ref{fig:GenusOneTrisections}. From here, it can be seen that $\gamma_1^i$ and  $\gamma_2^i$ are related by a Dehn twist about a curve bounding a disk in $H_\alpha$ so that after pushing them into $H_{\alpha}$ they become isotopic.

We may now take a diffeomorphism of the surface and perform handle slides of the $\alpha$ and $\beta$ curves so that the $\alpha$ and $\beta$ curves form the standard Heegaard diagram for $\#^{k}S^1 \times S^2$. By pushing the $g-k$ $\gamma_1$ and $\gamma_2$ curves dual to $\alpha$ curves into $H_\alpha$, and giving them the surface framing, we obtain framed links $L_1$ and $L_2$ in $H_\alpha \cup H_\beta$. On page 3104 of \cite{GK}, it is observed that the $L_i$ the are the framed attaching link for the 2-handles in a handle decomposition for $M_i$. Note that $g-k-1$ of these curves are identical, and the final curves have been shown to be isotopic in $H_\alpha$, which completes the argument. 

\end{proof}

\begin{remark}
\label{rem:unbalancedDistanceOne}
Note that the construction of $D^P(T_1,T_2)$ can be generalized to encompass unbalanced trisections where one of the $k_i$ agree on each trisection. We may then mimic the proof of Theorem \ref{thm:PantsDistanceOne} to study adjacent manifolds represented by unbalanced trisections. The proof goes through unchanged except that we must also consider the possibility that $\alpha^i,$ $\gamma_1^i$ and $\gamma_2^i$ form the unbalanced trisection diagram for $S^4$ shown in Figure \ref{fig:GenusOneTrisections}. In this case, the $\gamma$ curve parallel to the $\alpha$ curve does not play a role in the induced Kirby diagram whereas the $\gamma$ curve dual to the $\alpha$ curve manifests itself as a 2-handle. We may therefore conclude that distance 1 in this more general case corresponds to either changing a handle framing by 1 (the balanced case) or adding or removing a 2-handle. 
\end{remark}

\begin{figure}

  \centering
    \includegraphics[scale=.3]{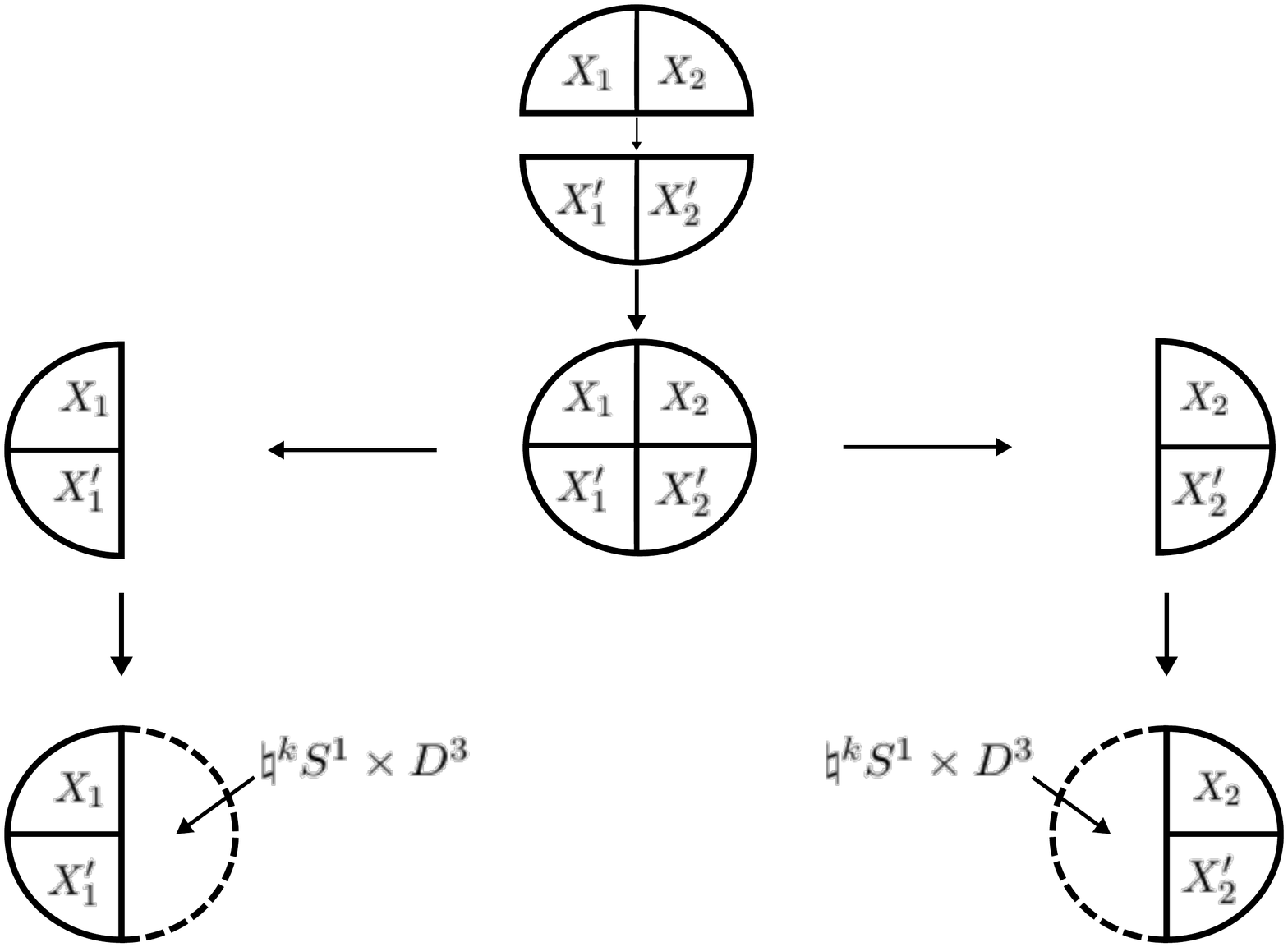}
    \caption{A schematic of the construction in Theorem \ref{thm:PantsDistanceOne}}
    \label{fig:DistanceOneArgument}
    
\end{figure}

We now seek to prove a partial converse to Theorem \ref{thm:PantsDistanceOne}. We begin by understanding how to obtain a trisection diagram from a Kirby diagram. To do this, we follow the proof for the existence of trisections given in \cite{GK}, while taking a little extra care to the particular diagram constructed. Take a Kirby diagram for $M$ with $k_1$ 1-handles and $k_2$ 2-handles. The 0 and 1-handles form $\natural^{k_1} S^1 \times D^3$ and have boundary $\#^{k_1}S^1 \times S^2$. We may take a genus $k_1$ Heegaard splitting for this boundary and draw $k_1$ parallel $\alpha$ and $\beta$ curves on the surface which bound disks in both handlebodies.  Now the framed attaching link for the 2-handles projects onto the Heegaard surface with perhaps a few crossings. Do Reidemeister 1 moves on the link on the surface to make the surface framing match the handle framing, and do a Reidemeister 2 move on each component to make sure it has at least 1 self crossing. Stabilize the Heegaard surface at each of the crossings to resolve them, resolving the self crossings as in Figures \ref{fig:ResolvingR2} and \ref{fig:ResolvingR1}. By construction, for each component of the link, $L_i$, we may choose a dual $\alpha$ curve, $\alpha_i$, that no other link component component intersects. Then we may slide any other $\alpha$ curve, $\alpha_j$, along $L_i$ over $\alpha_i$ to eliminate any intersections between $L_i$ and $\alpha_j$. 

Embedded on the Heegaard surface, we now see $g$ $\alpha$ curves, $g$ $\beta$ curves, and $k_2$ curves coming from the attaching link which are dual to $k_2$ $\alpha$ curves and disjoint from the rest of the $\alpha$ curves. We complete L to the set of $\gamma$ curves by adding in $g-k_2$ curves parallel to each $\alpha$ curve which does not intersect any component of $L$. It is clear that the pairs of curves $(\alpha, \beta)$ and $(\alpha, \gamma)$ are Heegaard diagrams for the connect sum of some number of copies of $S^1 \times S^2$. What is left to check is that the same holds for the pair $(\beta, \gamma)$.

The $\gamma$ curves define a handlebody, $H_\gamma$. Note that this handlebody is the result of pushing the $\gamma$ curves dual to $\alpha$ curves into $H_\alpha$ and performing surface framed Dehn surgery on them. But these dual curves come from the attaching link for the 2-handles of a closed 4-manifold. After attaching 2-handles along these curves pushed into $H_\alpha$, $H_\alpha$ becomes $H_\gamma$, but $H_\beta$ remains unchanged. Now $H_\gamma$ and $H_\beta$ form a Heegaard splitting for the boundary of the 3- and the 4-handles so that the pair  $(\beta, \gamma)$ is indeed a Heegaard diagram for some number of copies of $S^1 \times S^2$. We now have a possibly unbalanced trisection diagram for $M$, which we may balance by connect summing with the genus $1$ unbalanced trisection diagrams for $S^4$.

\begin{figure}
\centering
\includegraphics[scale=.5]{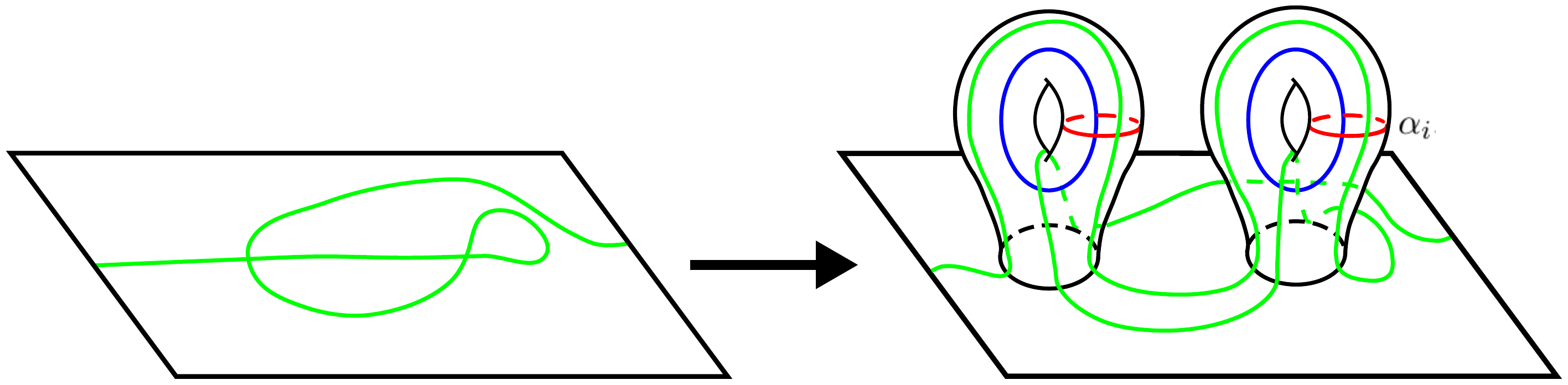}
\caption{Resolving a Reidemeister 2 move of the attaching link on the Heegaard Surface }
\label{fig:ResolvingR2}
\end{figure}

\framingPants

\begin{proof}
Let $L_i$ be the framed attaching links for $K_i$ and let $l_i$ be the component of the $L_i$ in which the framing differs. Without loss of generality, suppose that $|fr(l_1)| > |fr(l_2)|$ where $fr(l_i)$ is the framing of $l_i$. Since $K_1$ and $K_2$ have the same 0- and 1-handles, we may project both attaching links onto the Heegaard surface for the boundary of the union of the 0- and 1-handles. Introduce self intersections as previously described to obtain dual $\alpha$ curves, and to make the surface framing match the handle framing. This results in almost the same immersed curves on the Heegaard surface, except that $l_1$ has one more kink in it than $l_2$. Stabilize the Heegaard surface at all crossings, and in the extra kink, send $l_2$ over the stabilizing genus without twisting so as not to change the framing. See Figure \ref{fig:ResolvingR1} for an illustration of this process. 

We must now choose dual $\alpha$ curves for each component of the link in order to eliminate intersections. Let $\alpha_i$ be the $\alpha$ curve in the stabilization where $l_1$ and $l_2$ differ. Choose $\alpha_i$ to be the $\alpha$ curve dual to both $l_1$ and $l_2$ and choose arbitrary dual $\alpha$ curves for the rest of the components of the $L_i$. We now claim that eliminating the extra $\alpha$ intersections with $L_i$ by sliding curves off over the dual $\alpha$ curves along arcs parallel to the link components results in identical $\alpha$ curves. Sliding any curve along a link component which is not $l_1$ or $l_2$ obviously results in the same curve since we have constructed these curves to be identical. Moreover, sliding an $\alpha$ curve over $\alpha_i$ along $l_1$ is isotopic to sliding the $\alpha$ curve over $\alpha_i$ along $l_2$ as can be seen in Figure \ref{fig:HandleSlideAnnulusEquivalence}.

We now have possibly unbalanced trisection diagrams for $M_1$ and $M_2$ with identical $\alpha$ and $\beta$ curves. We seek to show that the $k_i$ for both of these manifolds are equal so that we may connect sum with the same unbalanced trisections of $S^4$ in order to balance them. It is straightforward to show that a $(g;k_1,k_2,k_3)$-trisection has Euler characteristic $2+g-k_1-k_2-k_3$. First note that both of these trisections have the same genus. Furthermore, $k_1$ is the number of copies of $S^1 \times S^2$ formed by the $\alpha$ and $\beta$ curves, which is clearly the same for both trisections. In addition, $k_3$ comes from the $\alpha$ and $\gamma$ curves, which we have constructed to be the same in both cases. Finally, the assumption that $M_1$ and $M_2$ have the same Euler characteristic ensures that these manifolds have equal $k_3$ so that we may balance these trisections in an identical manner. 

Finally, we complete both sets of $\gamma$ curves to pants decompositions of the handlebodies to finish the argument. To this end, note that $l_1$ and $l_2$ intersect transversely in one point so that the boundary of a regular neighborhood of the curves bounds a disk in both handlebodies. This cuts off a punctured torus containing $l_1$ and $l_2$. Outside of this punctured torus the $\gamma$ handlebodies are identical and so we may complete them to an arbitrary pants decomposition. The resulting pants decompositions are easily seen to be one apart in the pants complex.
\end{proof}

\begin{figure}
\centering
\includegraphics[scale=.5]{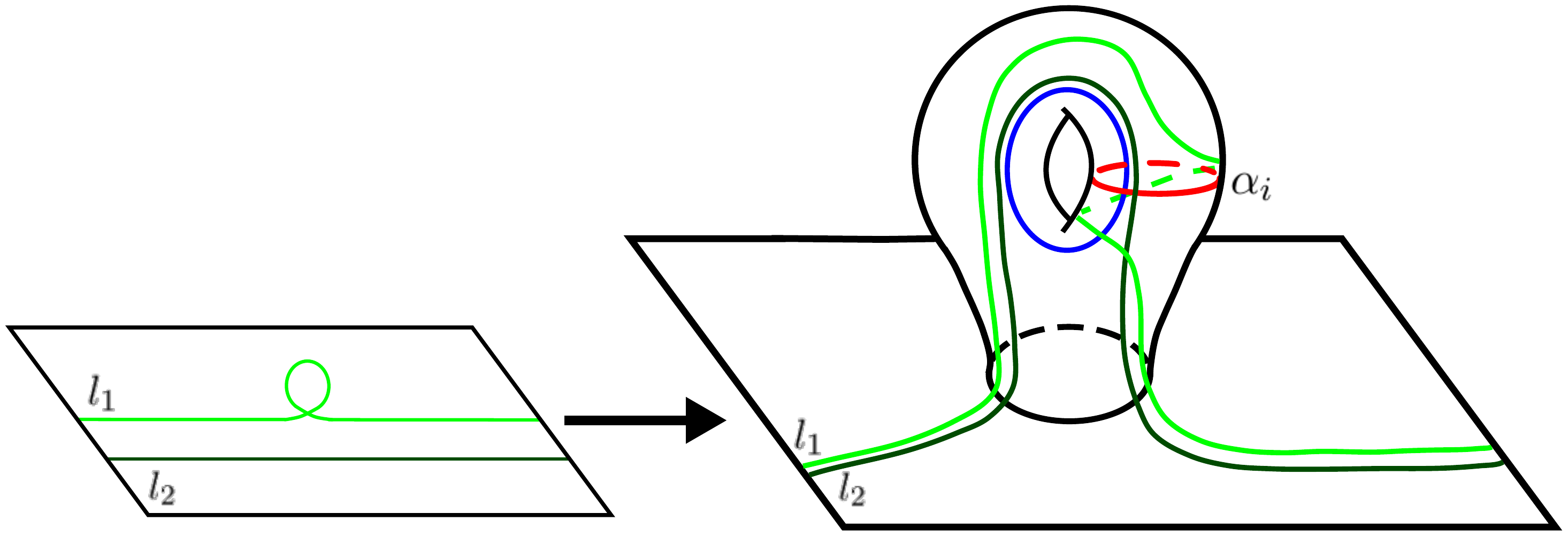}
\caption{Resolving a Reidemeister 1 move to change the surface framing of $l_1$ by 1. Parallel curves such as $l_2$ can be sent over the stabilizing surface without twisting about $\alpha_i$ to preserve the surface framing}
\label{fig:ResolvingR1}
\end{figure}

\begin{figure}
\centering
\includegraphics[scale=.3]{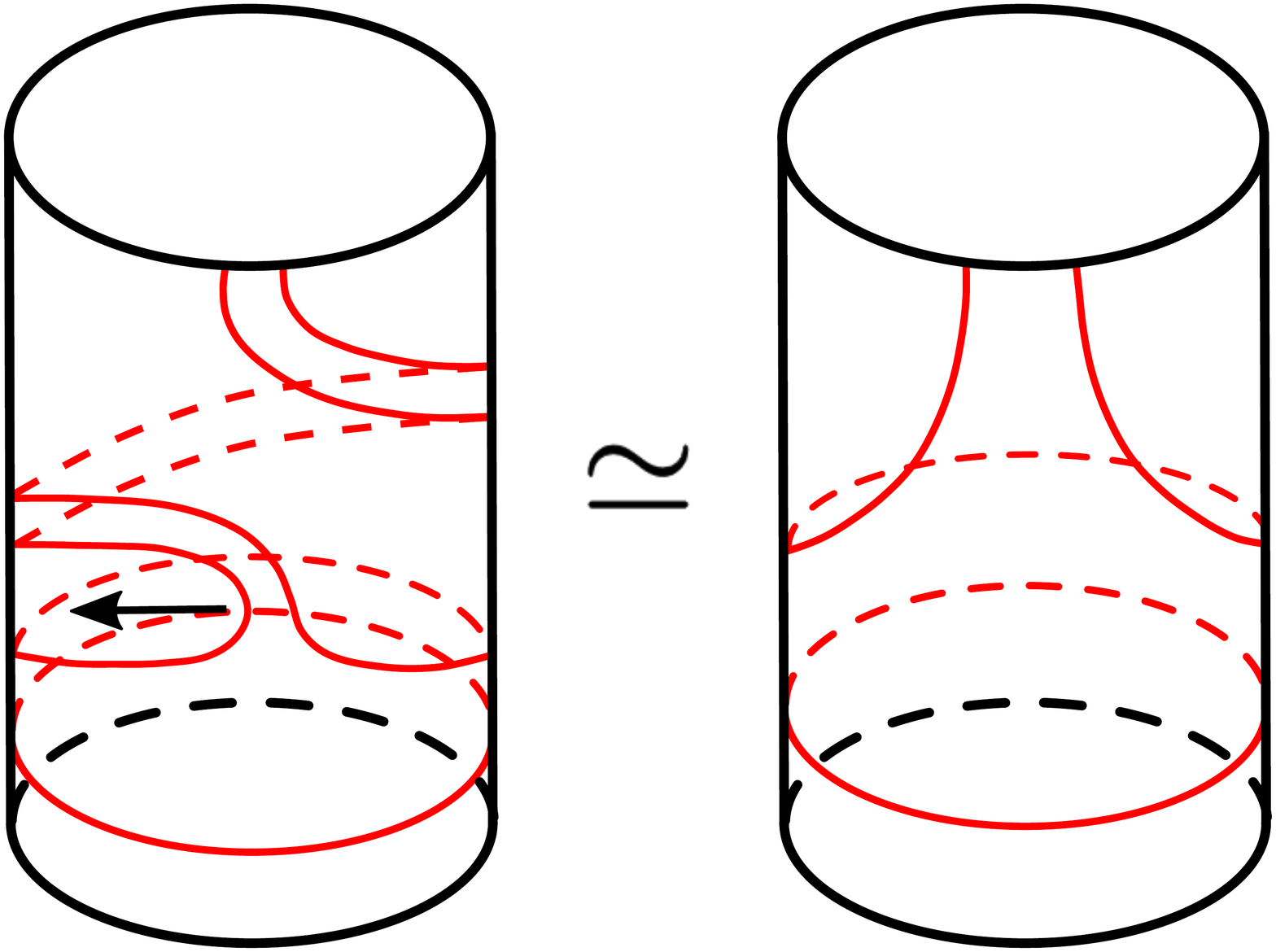}
\caption{Dehn twisting the sliding arc about the target curve does not change the isotopy type of slid curve.}
\label{fig:HandleSlideAnnulusEquivalence}
\end{figure}

%

We may also alter the framings of 2-handles by Dehn twisting about a chosen dual $\alpha$ curve which intersects no other component of the attaching link (recall that such curves may always be created by the introduction of self crossings in the link component). The result of repeatedly Dehn twisting a link component about the given $\alpha$ curve may intersect our original link component many times, however both curves lie in the punctured torus filled by the $\alpha$ curve and the original link component. In addition, adding more Dehn twists to the sliding arc in Figure \ref{fig:HandleSlideAnnulusEquivalence} does not change the isotopy type of the resulting curve so that we may again eliminate intersections via isotopic handle slides of the $\alpha$ curves. These are all the essential ingredients to the following theorem, whose details we leave to the reader.

\begin{theorem}
\label{prop:FramingDual}
Let $M_1$ and $M_2$ be non-diffeomorphic 4-manifolds which have Kirby diagrams, $K_1$ and $K_2$ respectively. If $K_1$ and $K_2$ only differ in the framing of some 2-handle, then $D(M_1,M_2) = 1$.
\end{theorem}

\section{Complexes of Trisections}

We next seek to define a collection of graphs associated to trisections. Here, it is useful to consider the more general case of unbalanced trisections. Fix a surface, $\Sigma$, and two handlebodies, $H_\alpha$ and $H_\beta$, with boundary $\Sigma$ so that $H_\alpha \cup_\Sigma H_\beta \cong \#^{k_1}S^1 \times S^2$.  We may identify the first two handlebodies in a $(g;k_1,k_2,k_3)$-trisection with  $H_\alpha \cup_\Sigma H_\beta$. The third handlebody then gives rise to some handlebody subset of $P(\Sigma)$. We therefore have a subcomplex of the pants complex associated to any (possibly unbalanced) trisection with parameters $(g;k_1,-,-)$. This motivates the following definition:

\begin{definition}
Fix a genus g surface surface $\Sigma$ and two handlebodies $H_\alpha$ and $H_\beta$ so that $H_\alpha \cup_\Sigma H_\beta \cong \#^{k_1}S^1 \times S^2$. Define the $(g;k_1,-,-)$ complex of trisections, $\mathbf{P(g,k_1)}$, to be the full subgraph of the pants complex whose vertices are  $\{\gamma \in P(\Sigma) |$ $\gamma$ defines $H_\gamma,$ $H_{\alpha} \cup_\Sigma  H_{\gamma} \cong \#^{k_2} S^1 \times S^2,$ and $ H_{\beta} \cup_\Sigma  H_{\gamma} \cong \#^{k_3} S^1 \times S^2 \}$.
\end{definition}

\begin{definition}
$\gamma \in P(g, k_1)$ is a \textbf{representative} for a trisection $T$ if $\gamma$ defines $H_\gamma$ and $H_\alpha \cup H_\beta \cup H_\gamma$ is a spine for $T$. We say $T_1$ and $T_2$ are \textbf{adjacent} in $P(g,k_1)$ if they have representatives which are adjacent.
\end{definition}

Note that a trisection has many representatives in $P(g,k_1)$. Not only are there infinitely many vertices in the pants complex defining the same handlebody, but multiple different handlebodies may represent the same trisection. For example, if $k_1>0$ there is some nonseparating curve which bounds disks in both $H_\alpha$ and $H_\beta$. A Dehn twist about this curve will usually change $H_\gamma$, but will give rise to a diffeomorphic trisection. More generally, we could take any element of the mapping class group $Mod(H_\alpha \cup H_\beta, \Sigma )$ which does not extend across $H_\gamma$ to produce similar results.

\begin{lemma}
\label{lem:2stabToCP2}
Let $T$ be a stabilized trisection of $M^4$. Then there exists a trisection $T'$ for $M \# \mathbb{C}P^2$ so that $T$ and $T'$ are adjacent in $P(g,k_1)$.
\end{lemma}

The proof of the previous lemma is straightforward. We may in fact weaken the hypothesis that $T$ is stabilized to the condition that $T$ is 2 or 3-stabilized (i.e. has a connect summand which is a $(1;0,1,0)$ or $(1;0,0,1)$ trisection of $S^4$) but the lemma as stated will be sufficient for our needs. This lemma is useful to us because 4-manifolds can change drastically under connect sums with $\mathbb{C}P^2$ and $\overline{\mathbb{C}P^2}$. The following corollary of Wall's theorem in \cite{CW} makes this precise.

\begin{proposition} (Corollary 9.1.14 of \cite{GS})
\label{prop:WallStability}
Let $M_1$ and $M_2$ be simply connected 4-manifolds. Then there exist natural numbers $l_1,l_2,m_1,m_2$ so that $M_1 \#^{l_1} \mathbb{C}P^2 \#^{m_1} \overline{\mathbb{C}P^2}$ is diffeomorphic to $M_2 \#^{l_2} \mathbb{C}P^2 \#^{m_2} \overline{\mathbb{C}P^2}$
\end{proposition}

We are now well equipped to prove the main proposition of this section.

\begin{proposition}
Let $M_1$ and $M_2$ be simply connected 4-manifolds. Then there exists a natural number, $k$, and $(g,k,-,-)$-trisections, $T_1$ and $T_2$, for $M_1$ and $M_2$ respectively, so that $T_1$ and $T_2$ are in the same connected component of $P(g,k)$. 
\end{proposition}

\begin{proof}
Take arbitrary trisections $T_1$ of $M_1$, and $T_2$ of $M_2$. Now 1- and 2-stabilize them so that they have the same genus, $g$, and the same $k_1$. We will first calculate the number of stabilizations needed for the construction. Let $l_1,l_2,m_1,m_2$ be as in Proposition \ref{prop:WallStability}. Let $a=max\{l_1 + m_1, l_2+m_2\}$. After 2-stabilizing $T_1$ and $T_2$ $a$ times, we may change each 2-stabilization into a summand of $\mathbb{C}P^2$ or $\overline{\mathbb{C}P^2}$ to obtain two (possibly different) trisections for the same 4-manifold. By Theorem \ref{thm:TrisectionsStabilize}, we may perform some number of balanced stabilizations on the resulting trisections until they are isotopic. Let $b$ be the number of stabilizations needed to make the resulting trisections isotopic.

We claim that $T_1^{g+a+3b}$ and $T_2^{g+a+3b}$ can be connected in $P(g+a+3b,k_1+b)$. To see this, observe that by Lemma \ref{lem:2stabToCP2}, each 2-stabilization can be changed into an extra factor of  $\mathbb{C}P^2$ or $\overline{\mathbb{C}P^2}$ adjacent to $T_1^{g+a+3b}$ or $T_2^{g+a+3b}$. Changing each 2-stabilization in $T_1^{g+a+3b}$ to the appropriate $\mathbb{C}P^2$ or $\overline{\mathbb{C}P^2}$ summand successively leads to a path to a trisection of $M_1 \#^{l_1} \mathbb{C}P^2 \#^{m_1} \overline{\mathbb{C}P^2}$ which we know to be diffeomorphic to $M_2 \#^{l_2} \mathbb{C}P^2 \#^{m_2} \overline{\mathbb{C}P^2}$. Moreover, the constructed trisections have been stabilized enough to become isotopic.
\end{proof}

\begin{figure}
\centering
\includegraphics[scale=.4]{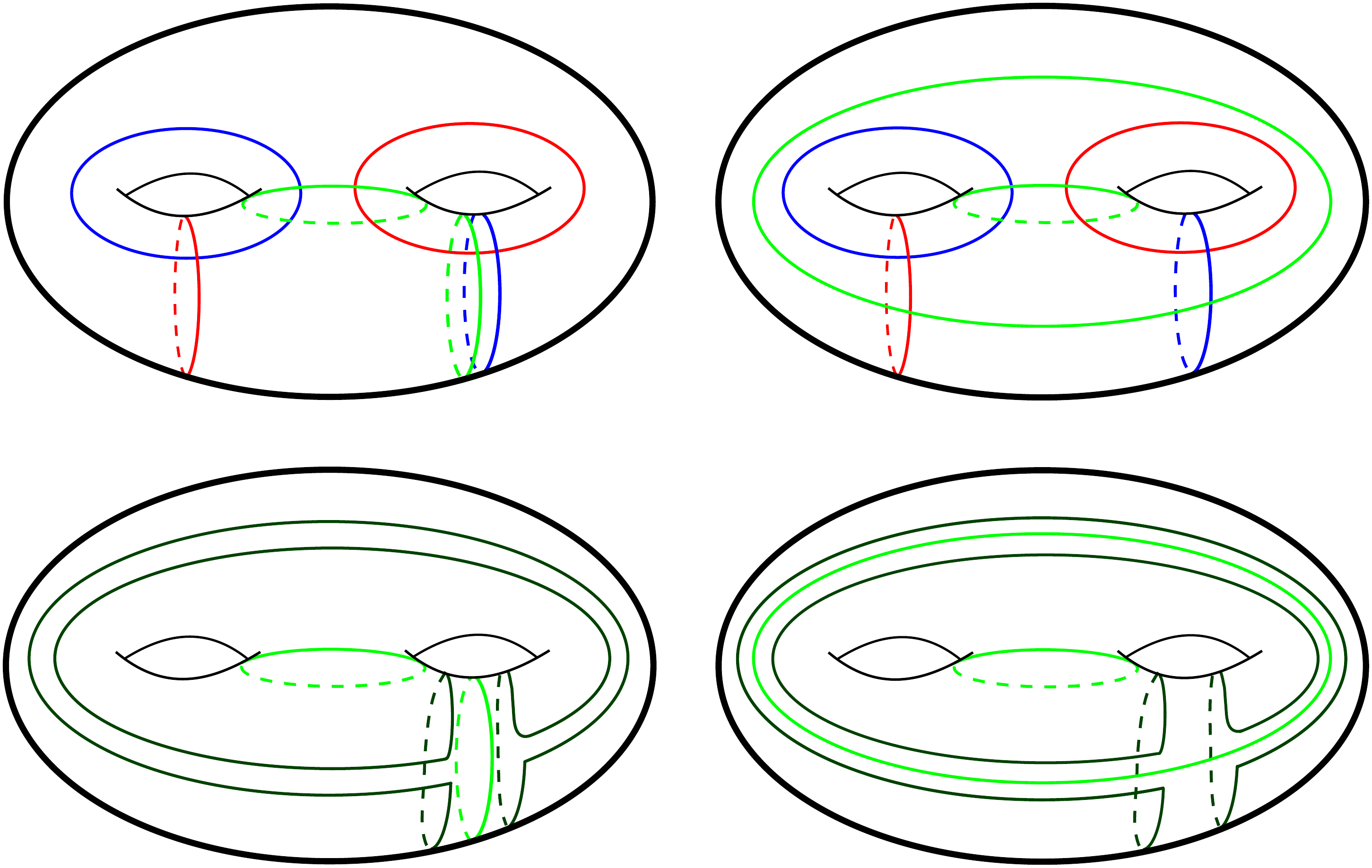}
\caption{$S^4$ (shown on the left) is adjacent to $S^2\times S^2$ (shown on the right) in $P(2,0)$. }
\label{fig:D(S2xS2,S4)=1}
\end{figure}

It is especially interesting to know which manifolds are adjacent to $S^4$, for if $N$ is adjacent to $S^4$, then for any M, we may stabilize a trisection to get an adjacent trisection for $M\#N$. It is straightforward to see that $S^4$ is adjacent to $\mathbb{C}P^2$, $\overline{\mathbb{C}P^2}$ and $S^1 \times S^3$. Furthermore, Figure \ref{fig:D(S2xS2,S4)=1} shows that $S^4$ is also adjacent to $S^2 \times S^2$. It is tempting to believe that this is a complete list of manifolds. In light of Remark \ref{rem:unbalancedDistanceOne}, manifolds adjacent to $S^4$ correspond to starting with some (perhaps very complicated) Kirby diagram for $S^4$ and then changing the framing of some 2-handle, or adding/removing a 2 handle. We conclude with a question.

\begin{question} Which 4-manifolds are adjacent to $S^4$?

\end{question}

\bibliography{mybib}{}

\bibliographystyle{plain}

\end{document}